\documentclass[a4,11pt]{article}
\usepackage[english]{babel}
\usepackage{epsfig}
\usepackage{amsmath,amsthm,amsfonts,amssymb,amscd}
\usepackage{latexsym}
\usepackage{xypic}
\usepackage{makeidx}
\usepackage{ulem}

\newtheorem{definition}{Definition}
\newtheorem{hip}{Hipótese}
\newtheorem{prop}{Proposition}
\newtheorem{theorem}{Theorem}

\newtheorem{lemma}{Lemma}
\newtheorem{obs}{Observation}
\newcommand{\R}{\mathbb{R}}

\begin{document}

\title{ \Large{\bf Consistent Approximations to Impulsive Optimal Control Problems}}

\author{ {\bf {\large \underline{Daniella Porto}}}, \hspace*{1cm} {\bf {\large
 Geraldo Nunes Silva}}, \hspace*{1cm} {\bf Helo\'isa Helena Marino Silva}  \\
 {\small Depto. de Matemática Aplicada, IBILCE, UNESP,} \\
 {\small 15054-000, S\~ao Jos\'e do Rio Preto, SP} \\
 {\small E-mail: danielinha.dani@gmail.com, \hspace*{.2cm}  gnsilva03@gmail.com, \hspace*{.2cm} hsilva@ibilce.unesp.br}}

\date{}

\maketitle

\thispagestyle{empty}

\noindent{\bf Abstract:} {\it We analyse the theory of consistent approximations given by Polak in \cite{Polak}, and we use it in an impulsive optimal control problem. We reparametrize the original system and build consistent approximations for this new reparametrized problem. So, we prove that if a sequence of solution of the consistent approximations is converging, it will converge to a solution of the reparametrized problem, and, finally, we show that from a solution of the reparametrized problem we can find a solution of the original one.}

\section{Introduction}

In this paper, we study impulsive optimal control problems in which we apply a theory called consistent approximations. This theory was introduced in \cite{El1992}, \cite{Polak}, and uses approximated problems of finite dimension. From an infinite dimension problem $(P)$, we can build a sequence of problems $(P_N)$, where these problems have finite dimension and epi-converge (convergence between the epigraphs) to $(P)$. This convergence ensures that every sequence of global or local minimum of $(P_N)$ that converges, will converge to a global or local minimum of $(P)$, respectively. It is necessary to use optimality functions to represent the first-order necessary conditions because for optimal control with state and control constraints, that are complex, it is easier to work with optimality functions than with classical forms as first-order necessary condition. In \cite{Polak} is given an application of this theory by using an optimal control problem.

There exist many papers where impulsive control systems are studied, we cite, for example, \cite{SV:1996}, \cite{Peter}. The article \cite{SV:1996} shows that the solution set of an impulsive system, given by differential inclusion, is weakly$^{*}$ closed and the article \cite{Peter} builds a numerical approximation for the impulsive system, also given by differential inclusion, using the Euler's discretization. It is shown that there is a subsequence of the solution sequence obtained by this discretization that graph-converges to a solution of the original system. In this paper we propose an approximation by absolutely continuous measures, using the convergence of graph-measure.

There exist a great number of papers that talk about impulsive optimal control problems where the control systems involve measures, and that discuss theoretical conditions for optimality of control process. We can cite, for example, \cite{karamzin}, \cite{K:2005}, \cite{motta_caterina:2011}, \cite{PS:2000}, \cite{SLRB:2000}, \cite{SR:1996}, \cite{SV:1997}. On the other hand, the literature about numerical methods for impulsive optimal control problems is rather scarce, \cite{MR1816510}. In this paper, a space-time reparametrization of the impulsive problem is adopted, an approximation scheme for that augmented system is construct and it is proved that such approximation converge to the value function of the augmented problem. Finally, the sequence of discrete optimal controls converge to an optimal control for the continuous problem.

Regarding usual optimal control problems, there are some works that aim to solve them using discrete approximation by Euler \cite{DH:2001}, \cite{DH:2000c}, or Runge Kutta \cite{DH:2000RK}, \cite{WH:2000b}, \cite{WH:2000}. The scheme used is 1) discretize the optimal control problem and 2) solve the resulting nonlinear optimization problem. The choice of the method of resolution depends on the structure of the optimal control problem and personal preferences. Among the several proposals for solutions of nonlinear optimization problems arising from discretization, we cite some more recent \cite{KM:2010} and \cite{KM:2007}.

This work aims to contribute with the presentation of the Euler's method application for impulsive optimal control problems. We show that an impulsive optimal control problem can be reparametrized and discretized by Euler's method to generate a subsequence of optimal trajectories of Euler that converge to an optimal trajectory of the reparametrized problem, using an appropriate metric. From that we can find the optimal solution to the continuous problem. We are given a generalization of valid results for non impulsive optimal control problems, \cite{DH:2001}.

In Section 2 we introduce the theory of consistent approximations. We define the impulsive system and get the reparametrized system in Section 3. In Section 4 is defined the impulsive optimal control problem. We establish the approximated problems to our reparametrized optimal control problem in Section 5, and finally, the consistent approximations is given in Section 6. We also show the convergence between a sequence of global or local minimum to the approximated problems and the local or global minimum to the reparametrized problem in the same section. Section 7 has bounds on approximations errors and Section 8 conclusions.

\section{Theory of Consistent Approximations}\label{2.1}
Let $B$ be a normed space. Consider the problem
$$
(P) \;\;\;\; \min_{x \in S_C} f(x),
$$
where $f:B\rightarrow \mathbb{R}$ is continuous and $S_C \subset B$.

Let $\mathcal{N}$ be an infinite subset of $\mathbb{N}$ and $\{S_N\}_{N \in \mathcal{N}}$ be a family of finite dimension subspaces of $B$ such that $S_{N_1} \subset S_{N_2}$ if $N_1 < N_2$ and
$\cup S_N$ is dense in $B$. For all $N \in \mathcal{N}$, let $f_N: S_N \rightarrow \mathbb{R}$ be a continuous function that approximates $f(\cdot)$ over $S_N$, and let $S_{C,N} \subset S_N$ be an approximation of $S_C$. Consider the approximated problems family
$$
(P_N) \;\;\;\; \min_{x \in S_{C,N}} f_N(x), \;\; N \in \mathcal{N}.
$$

Define the epigraphs associated to $(P)$ and $(P_N)$, respectively, as
$$
E:=\{(x,r): x \in S_C, f(x)\leq r\}
$$
and 
$$
E_N:=\{(x,r): x \in S_{C,N}, f_N(x)\leq r\}.
$$
Note that the problems above can be rewritten like
$$
(P) \;\; \min_{(x,r) \in E}r \quad\quad\qquad \qquad (P_N) \;\; \min_{(x,r) \in E_N}r, \; N \in \mathcal{N}.
$$
Note that the unique difference between $(P_N)$ and $(P)$ are the epigraphs of them and if  $E_N$ converges to $E$, in the sense of ``Kuratowski'', we have a convergence between the problems. In Theorem 3.3.2, \cite{Polak}, the epigraph convergence described is equivalent to the items a) and b) of the next definition about consistent approximations.

\begin{definition}
Let the functions $f(\cdot)$ and $f_N(\cdot)$ and the sets $B$, $S_C$, $S_N$ and $S_{C,N}$ be defined as above.\\
i) We say $P_N$ epi-converge to $P$ ($P_N \rightarrow^{Epi} P$) if:\\
a) For all $x \in S_C$ there exists a sequence $\{x_N\}_{N \in \mathcal{N}}$, with $x_N \in S_{C,N}$, such that $x_N \rightarrow^{\mathcal{N}} x$, with $N\rightarrow \infty$, and $\overline{\lim}_{N \rightarrow \infty}f_N(x_N)\leq f(x)$;\\
b) For all infinite sequence $\{x_N\}_{N \in \mathcal{K}}$, $\mathcal{K} \subset \mathcal{N}$, such that $x_N \in S_{C,N}$, for all $N \in \mathcal{K}$, and $x_N\stackrel{\mathcal{K}}{\rightarrow} x$, with $N \rightarrow \infty$, then $x \in S_C$ and $\underline{\lim}_{N \in \mathcal{K}}f_N(x_N)\geq f(x)$, as $N \rightarrow \infty$.\\
ii) We say the upper semicontinuous functions $\gamma_N: S_{C,N} \rightarrow \R$ are optimality functions for the problems $(P_N)$ if $\gamma_N(\eta) \leq 0$, $\forall \; \eta \in S_{C,N}$ and if $\hat{\eta}_N$ is a local minimizer of $(P_N)$ then $\gamma_N(\hat{\eta}_N)=0$. We can define the optimality function $\gamma:S_C\rightarrow \R$ for $(P)$ in the same way.\\
iii) The pairs $(P_N,\gamma_N)$ of the sequence $\{(P_N,\gamma_N)\}_{N \in \mathcal{N}}$ are consistent approximations to the pair $(P,\gamma)$ if $P_N$ epi-converge to $P$ and for all sequence $\{x_N\}_{N \in \mathcal{N}}$ where $x_N \in S_{C,N}$ and $x_N \rightarrow x \in S_C$ we have $\overline{\lim}_{N \rightarrow \infty}\gamma_N(x_N) \leq \gamma(x)$.
\end{definition}

The key of that epigraph convergence is given by Theorem 3.3.3, \cite{Polak}, where it is shown that if $(P_N)$ epi-converges to $(P)$ and if $\{x_N\}_{N \in \mathcal{N}}$ is a sequence of local or global solutions to $(P_N)$ so that $x_N$ converges to $x$, then $x$ is a local or global minimizer of $(P)$ and $f_N(x_N)$ converge to $f(x)$, $N\rightarrow \infty, N \in \mathcal{N}$.

It is necessary to define the optimality functions since the epi-convergence alone can not guarantee that the sequence of stationary points of $(P_N)$ converges to a stationary point of $(P)$, as we can observe in an example given by \cite{Polak}, page 397.

\section{The Impulsive System}\label{2.2}
Before we define the impulsive optimal control problem we define the impulsive system that is related to it and show some results given by \cite{Peter}.

Consider the impulsive system
\begin{equation}\label{eq}
\begin{cases}
dx =  f(x,u)dt + g(x)d\Omega, t\in [0,T],\\
x(0)=\xi^0 \in \mathcal{C},
\end{cases}
\end{equation}
where $f:\mathbb{R}^n \times \mathbb{R}^m \rightarrow \mathbb{R}^n$ is linear in $u$, $g:\mathbb{R}^n\rightarrow \mathcal{M}_{n \times q}$, where $\mathcal{M}_{n \times q}$ is the space of $n \times q$ matrices whose entries are real, $\mathcal{C}\subset \mathbb{R}^n$ is closed and convex, the function $u:[0,T] \rightarrow
\mathbb{R}^m$ is Borel measurable and essentially bounded, $\Omega := (\mu,|\nu|,\psi_{t_i})$ is the impulsive
control, where the first component $\mu$ is a vectorial Borel measure with range in a convex, closed cone $K\subset
\mathbb{R}^q_{+}$. The second component is such that there exists $\mu_N:[0,T] \rightarrow K$ so that $(\mu_N,|\mu_N|) \rightarrow^{*} (\mu,\nu)$. As $K\subset \R^q_{+}$ we must have $\nu=|\mu|$. The  functions $\psi_{t_i}:[0,1]\rightarrow K$ are associated to the measure atoms, that is, $\{\psi_{t_i}\}_{i \in \mathcal{I}}$ where $\mathcal{I}$ is the set of atomic index of the measure $\mu$ and we define $\Theta :=\{t_i\in[0,T]: \mu(\{t_i\})\neq 0\}$, where $\mu(t)$ is the vectorial value of the measure in $K$. The functions $\psi_{t_i}$ are measurable, essentially bounded and satisfy
\[\begin{array}{ll}
\mbox{i)} & \sum_{j=1}^q |\psi^j_{t_i}(\sigma)| = |\mu|(t_i) ~\mbox{a.e. } \sigma \in [0,1];\\ \mbox{ii)} & \int_{0}^{1}
\psi^j_{t_i}(s) ds = \mu^j(t_i), \quad j=1,2,\ldots,q,
\end{array}\]
for all $t_i \in \Theta$.

The functions $\psi_{t_i}(\cdot)$ give us information about the measure $\mu$ during the atomic time $t_i \in \Theta$.

\subsection{The Reparametrized Problem}
We obtain a reparametrized problem that is approximated by using the consistent approximations. This can be done, without loss of information, due to the theorem \ref{teo9} that is enunciated in this subsection and was proved in \cite{wol}. It says that the reparametrized problem and the original problem have equivalents solutions, up to reparametrization. For more information about it see \cite{SV:1996}, \cite{Z:2005}, \cite{wol}, \cite{Maso}.

Firstly, we study the impulsive system given by \eqref{eq}. For this, let $\Omega=(\mu,\nu,\{\psi_{t_i}\}_{t_i\in \Theta})$ be an impulsive control and $\xi^0 \in \mathbb{R}^n$ an arbitrary vector. Denote by $\mathcal{X}_{t_i}(\cdot;\xi^0)$ the solution to the system
$$
\begin{cases}
\dot{\mathcal{X}}_{t_i}(s)=g(\mathcal{X}_{t_i}(s))\psi_{t_i}(s), \; s \in [0,1],\\
\mathcal{X}_{t_i}(0)=\xi^0.
\end{cases}
$$

Consider
\begin{equation}\label{traj}
x_{\vartheta}:=(x(\cdot),\{\mathcal{X}_{t_i}(\cdot)\}_{t_i \in \Theta}),
\end{equation}
where $\vartheta:=(u,\Omega)$, $x(\cdot): [0,T]\rightarrow \mathbb{R}^n$ is a function of bounded variation with the discontinuity points in the set $\Theta$ and $\{\mathcal{X}_{t_i}(\cdot)\}_{t_i \in \Theta}$ is the collection of Lipschitz functions defined above. The definition of solution of the system \eqref{eq} is given in the sequence.

\begin{definition}\label{def2}
We say that $x_{\vartheta}$ is a solution of \eqref{eq} if
$$
x(t) = \xi^0 + \int_0^t f(x,u)d\sigma + \int_{[0,t]} g(x) d\mu_{c}+\sum_{t_i\leq t} [\mathcal{X}_{t_i}(1) - x(t_i-)] \; \forall t\in [0,T],
$$
where $\mu_c$ is the continuous component of $\mu$ and $x(t_i-)$ is the left-hand limit of $x(\cdot)$ in $t_i$.
\end{definition}

Now we do a time reparametrization and get a reparametrized system whose solution is equivalent to the solution of the original system \eqref{eq}, up to reparametrization. For this, define
$$
\pi(t):=\frac{t+|\mu|([0,t])}{T+\|\mu\|},\; t \in ]0,T],\; \pi(0-)=0.
$$
The last equality is a convention because $0$ can be an atom of $\mu$.

Then, there exists $\theta:[0,1] \rightarrow [0,T]$ such that
\begin{itemize}
\item $\theta(s)$ is non-decreasing;
\item $\theta(s)=t_i \;\; \forall t_i \in \Theta, \forall s \in I_i$, where $I_i=[\pi(t_i-),\pi(t_i)]$.
\end{itemize}

We define by $F(t;\mu):=\mu([0,t])$ if $t \in ]0,T]$, and $F(0;\mu)=0$ the distribution function of the measure $\mu$.

Let $\phi:[0,1] \rightarrow \mathbb{R}^q$ be given by
$$
\phi(s):=\begin{cases}
F(\theta(s);\mu) \; \mbox{if} \; s \in [0,1]\backslash(\cup_{i \in \mathcal{I}}I_i),\\
F(\theta(s);\mu) + \int_{[\pi(t_i-),s]}\frac{1}{\pi(t_i)-\pi(t_i-)}\psi_{t_i}(\alpha_{t_i}(\sigma))d\sigma \; \mbox{if} \; s \in I_i,
\end{cases}
$$
where $\alpha_{t_i}:[\pi(t_i-),\pi(t_i)]\rightarrow [0,1]$ is given by $\alpha_{t_i}(\sigma)=(\sigma-\pi(t_i-))/(\pi(t_i)-\pi(t_i-))$.

According to \cite{Maso}, $\theta(\cdot)$ and $\phi(\cdot)$ are Lipschitz. The Lipschitz constants are given by $b$ and $r$, respectively, and, furthermore, $\theta(\cdot)$ is absolutely continuous. We say $(\theta,\phi)$ is the {\bf graph completion} of the measure $\mu$.

With all the tools on hands we define a reparametrized solution of the system\eqref{eq}.

\begin{definition}\label{def1}
Let
\begin{equation}\label{defy}
y(s):=\begin{cases}
 x(\theta(s)) \qquad \mbox{if}\; s \in [0,1]\setminus\left(\cup_{i \in \mathcal{I}}I_i\right),\\
  \mathcal{X}_{t_i}(\alpha_{t_i}(s)) \;\; \mbox{if} \; s \in I_i, \; \mbox{for some} \; i \in \mathcal{I}.
      \end{cases}
\end{equation}
Then $y_{\vartheta}:=y$ is a reparametrized solution of \eqref{eq} since $y(\cdot)$ is Lipschitz in $[0,1]$ and satisfies
\begin{equation}\label{defy.}
\begin{cases}
\dot{y}(s) = f(y(s),u(\theta(s)))\dot{\theta}(s) + g(y(s))\dot{\phi}(s)\; \mbox{a.e}.\; s \in [0,1],\\
y(0)=\xi^0.
\end{cases}
\end{equation}
\end{definition}

The next theorem is proved in \cite{wol}.

\begin{theorem}\label{teo9}
Suppose that the impulsive control $\Omega$ is given and $x_{\vartheta}$ is as defined in $\eqref{traj}$. Then, $y_{\vartheta}$ is a reparametrized solution of $\eqref{eq}$ if and only if $x_{\vartheta}$ is a solution of $\eqref{eq}$.
\end{theorem}

\section{The Impulsive Optimal Control Problem}\label{2.3}

We need to describe the constrains on the control $u$. We are following \cite{Polak}. For this, denote by $L_2^m[0,T]$ the set of all functions defined from $[0,T]$ into $\mathbb{R}^m$ that have integrable square.

Let $\beta_{\max} \in \; (0,+\infty)$ be such that every control belongs to the ball $B(0,\beta_{\max}):=\{u \in \mathbb{R}^m; \|u\|_{\infty}\leq\beta_{\max}\}$.

Define 
$$
\hat{\mathcal{U}}:=\{u \in L_{\infty,2}^m[0,T]; \|u\|_{\infty} \leq \omega\beta_{\max}\},
$$
where $\omega \in \; (0,1)$ and $L_{\infty,2}^m[0,T]$ is the set of all functions defined from $[0,T]$ into $\mathbb{R}^m$ that are essentially bounded and the $L_2$ norm is considered over it.

Now, we define the set of constraints of the control $u$ by
$$
\mathcal{U}:=\{u \in \hat{\mathcal{U}}; u(t) \in \bar{\mathcal{U}} \subset B(0,\omega\beta_{\max}) \;\; \mbox{a.e.} \;\; t \in [0,T]\},
$$
where $\bar{\mathcal{U}} \subset \mathbb{R}^m$ is a convex, compact subset of the ball $B(0,\omega\beta_{\max})$.

Consider the impulsive optimal control problem
\[\begin{array}{lllll}
&\min &f^0(x(0),x(T))\\
(P)&&dx =  f(x,u)dt + g(x)d\Omega \; \mbox{a.e.}\; t\in [0,T],\\
&& x(0) \in  \mathcal{C},\; u \in \mathcal{U},\; \mbox{gc}\sup_{t \in [0,T]}|x(t)|\leq L,
\end{array}\]
where $f^0:\mathbb{R}^n \times \mathbb{R}^n \rightarrow \mathbb{R}$ is continuous, $L>0$ is given and the other functions and sets are defined as above. Here,
$$
\mbox{gc}\sup_{t \in [0,T]}|x(t)|=\sup_{s \in [0,1]}|y(s)|.
$$

We need the following assumption.

\begin{hip}\label{assump}
a) The functions $f(\cdot,\cdot)$ and $g(\cdot)$ are $C^1$, and there exist constants $K^{'}, K^{''} \in [1,\infty[$ such that, for all $x, \hat{x} \in \mathbb{R}^n$ and $u, \hat{u} \in B(0,\beta_{\max})$ we have
$$
|f(x,u)-f(\hat{x},\hat{u})|\leq K^{'}[|x-\hat{x}|+|u-\hat{u}|],
$$
$$
\|g(x)-g(\hat{x})\| \leq K^{''}|x-\hat{x}|,
$$
and $f(\cdot,\cdot)$ and $g(\cdot)$ have linear growth, that is, there exists a constant $K_1 < \infty$ so that
$$
|f(x,u)| \leq K_1(1+|x|) \;\; \mbox{e} \;\; \|g(x)\| \leq K_1(1+|x|).
$$
b) The function $f^0(\cdot,\cdot)$ is Lipschitz, has first derivative Lipschitz and is $C^1$ over bounded sets.\\
c) The impulsive system given by
\begin{equation}\label{system1}
\begin{array}{lllll}
&&dx =  f(x,u)dt + g(x)d\Omega \; \mbox{a.e.}\; t\in [0,T],\\
&& x(0)=\xi^0 \in  \mathcal{C},\; u \in \mathcal{U},\; \mbox{gc}\sup_{t \in [0,T]}|x(t)|\leq L,
\end{array}
\end{equation}
where all the variables are like above, is controllable.
\end{hip}

Let $(\xi_0,\xi_1) \in \mathcal{C}\times \R^n$ be arbitrarily chosen. We say an impulsive system like \eqref{system1} is {\bf controllable} if there exist a control $u \in \mathcal{U}$ and an impulsive control $\Omega$ so that the trajectory related to such control $x_{\vartheta}(\cdot)$ satisfies $x(0)=\xi_0$ and $x(T)=\xi_1$ and $\mbox{gc}\sup_{t \in [0,T]}|x(t)|\leq L$.  

Suppose \eqref{system1} is controllable, then if we arbitrarily choose $(\xi_0,\xi_1) \in \mathcal{C}\times \R^n$, there exists a trajectory $x_{\vartheta}(\cdot)$ of \eqref{system1} satisfying $x(0)=\xi_0$ and $x(T)=\xi_1$. We know there exists a solution of the reparametrized system \eqref{defy.} given by $y(\cdot)$ so that $y(0)=\xi_0$ and $y(1)=x_{\vartheta}(T)$, that is, the reparametrized system is controllable. 

We want to get the reparametrized impulsive optimal control problem. For this, it is necessary to define the constraints of the control $u\circ \theta$.

We define the set of constraints of the control $u\circ \theta$ by
$$
\mathcal{U}_C:=\{\hat{u} \in \hat{\mathcal{U}}_1; \hat{u}(s) \in \bar{\mathcal{U}} \subset B(0,\omega\beta_{\max}),\; \mbox{a.e.} \; s \in [0,1]\},
$$
where $\beta_{\max}$,  $\bar{\mathcal{U}}$ and $\omega$ are the same and $\hat{\mathcal{U}}_1:=\{\hat{u} \in L_{\infty,2}^m[0,1];\|\hat{u}\|_{\infty} \leq \omega\beta_{\max}\}$.

Define
$$
\tilde{S}_C:= \mathcal{C} \times \mathcal{U}_C \times \mathcal{P},
$$
whose $\mathcal{U}_C$ is as defined above and $\mathcal{P}$ is the set of all $\Omega:=(\mu,|\nu|,\{\psi_{t_i}\})$ that satisfies the assumptions of the system \eqref{eq}. We also define 
$$
S_C:=\{\eta \in \tilde{S}_C:\sup_{s \in [0,1]}|y^{\eta}(s)|\leq L\}.
$$
We represent by $y^{\eta}(\cdot)$ the solution of the system \eqref{defy.} for each $\eta \in \tilde{S}_C$.

We obtain the following reparametrized problem
$$
(P_{rep}) \; \min_{\eta \in S_C}f^0(y^{\eta}(0),y^{\eta}(1)).
$$

Note that $(P)$ and $(P_{rep})$ have the same solution, up to a reparametrization, because the objective function is the same. So, we will get the consistent approximations for $(P_{rep})$.

The theorem below guarantees that the system \eqref{defy.} has an unique solution.

\begin{theorem}\label{unicidade} Suppose $\eta=(\xi^0,u,\Omega)$ is given, where $\xi^0 \in \mathcal{C}, u \in L_{\infty,2}^m[0,1]$ and $\Omega=(\mu,|\nu|,\psi_{t_i})$ satisfy the assumptions of the system \eqref{eq}. Then, the system defined in \eqref{defy.} has an unique solution.
\end{theorem}

\begin{proof}
Suppose that $\eta$ is given and there exist two solutions denoted by $y^{\eta}_1$ and $y^{\eta}_2$. We have
\[\begin{array}{lllll}
\vspace{0.5cm}
|y^{\eta}_1(s)-y^{\eta}_2(s)| \leq \int_0^s\left(K^{'}|y^{\eta}_1(\sigma)-y^{\eta}_2(\sigma)||\dot{\theta}(\sigma)|+K^{''}|y^{\eta}_1(\sigma)-y^{\eta}_2(\sigma)||\dot{\phi}(\sigma)|\right)d\sigma\\
\vspace{0.5cm}
\qquad \qquad \qquad \;\;\leq \int_0^s |y^{\eta}_1(\sigma)-y^{\eta}_2(\sigma)|\left(K^{'}b+K^{''}r\right)d\sigma,\\
\end{array}\]
where in the first inequality we substituted the expressions of $y^{\eta}_1(\cdot)$ and $y^{\eta}_2(\cdot)$ given by the system \eqref{defy}, and then we used the Assumption \ref{assump}. By Bellman-Gronwall's Theorem, $|y^{\eta}_1(s)-y^{\eta}_2(s)|\leq 0$, that is, $y^{\eta}_1\equiv y^{\eta}_2$.
\end{proof}

\section{Approximated Problems}\label{3.1}
We need a metric over the space $S_C$. Consider $\Omega_1=(\mu_1,|\mu_1|,\psi_{t_i}^1)$, $\Omega_2=(\mu_2,|\mu_2|,\psi_{t_i}^2) \in \mathcal{P}$. We need to define a metric in the measure space $\mathcal{P}$. Consider the metric given by
$$
d_3(\Omega_1,\Omega_2)=d_4(\Omega_1,\Omega_2) + d_5(\Omega_1,\Omega_2),
$$
where $d_4(\cdot,\cdot)$ is a metric given by \cite{K:2005},
\[\begin{array}{lll}
\vspace{0.5cm}
d_4(\Omega_1,\Omega_2)=|(\mu_1,|\mu_1|)[0,T]-(\mu_2,|\mu_2|)[0,T]|\\
\vspace{0.5cm}
\qquad \qquad \;+\int_{0}^{T}|F_1(t;(\mu_1,|\mu_1|))-F_2(t;(\mu_2,|\mu_2|))|dt\\
\qquad \qquad \;+\max_{s \in [0,1]}|\phi_1(s)-\phi_2(s)|,
\end{array}\]
and $d_5(\cdot,\cdot)$ is related to the graph-convergence given by \cite{Peter},
$$
d_5(\Omega_1,\Omega_2)=\int_0^1|\dot{\theta}_1(s)-\dot{\theta}_2(s)|ds + \int_0^1|\dot{\phi}_1(s)-\dot{\phi}_2(s)|ds,
$$
with $(\theta_1,\phi_1)$ and $(\theta_2,\phi_2)$ the graph completion of $\mu_1$ and $\mu_2$, respectively.

According to \cite{K:2005}, the space $\mathcal{P}$ with the metric $d_4$ is a metric space, and, furthermore, is the completion of the absolutely continuous measures given over $[0,T]$ in the metric $d_4$.

Note that $S_C \subset \mathbb{R}^n \times L_{\infty,2}^m[0,1] \times \mathcal{P}=: B$. Define the metric $d$ over $B$ as
$$
d=d_1+d_2+d_3,
$$
where $d_3$ is given above and
$$d_1(\xi^0,\xi^1)=|\xi^0-\xi^1|_{\R^n} \;\; \mbox{and} \;\; d_2(u_1,u_2)=\int_0^1|u_1(s)-u_2(s)|^{2}_{\R^m}ds$$.

We want to get consistent approximations to the problem $(P_{rep})$. For this, define the sets
$$
\mathcal{N}:=\{2^k\}_{k=1}^{\infty} \;\; \mbox{and} \;\; S_N:=\mathcal{C}_N \times L_N^{m} \times \mathcal{P}_N \;\; \mbox{for all} \; N \in \mathcal{N},
$$
where $\mathcal{C}_N:=\mathbb{R}^n \;\; \forall$ $N \in \mathcal{N}$,
$$
L_N^m:= \{u_N \in L_{\infty,2}^m[0,1]; u_N(s)= \sum_{k=0}^{N-1}u_k\tau_{N,k}(s)\},
$$
with $u_k \in \mathbb{R}^m$ and
$$
\tau_{N,k}(s):=\begin{cases}
 1 \;\;\; \forall \; s \in [k/N,(k+1)/N[ \;\; \mbox{if}\;\; k\leq N-2,\\
 1 \;\;\; \forall \; s \in [k/N,(k+1)/N] \;\; \mbox{if}\;\; k= N-1,\\
 0 \;\;\; \mbox{otherwise}.
      \end{cases}
$$
and $\mathcal{P}_N$ is given by
$$
\mathcal{P}_N:=\{(\mu_N,|\mu_N|,0): \mu_N([0,t]):=F_N(t)\},
$$
where $|\mu_N|$ is the variation of the measure $\mu_N$, $F_N(0)=0$ and over $]0,T]$,
$$
F_N(t):=\sum_{k=0}^{N-1}\bar{\tau}_{N,k}(t),
$$
$$
\bar{\tau}_{N,k}(t):=\begin{cases}
 b_k+\frac{t-\bar{t}_k}{\bar{t}_{k+1}-\bar{t}_k}(b_{k+1}-b_k)\;\; \forall \; t \in [\bar{t}_k,\bar{t}_{k+1}],\\
 \quad k=0,...,N-1, 0=\bar{t}_0<...<\bar{t}_N=T, \\
 0 \;\;\; \mbox{otherwise},
      \end{cases}
$$
with $b_k \in K$ for all $k=0,...,N-1$. Note that $\mu_N$ is an absolutely continuous measure from $[0,T]$ to $K$ ($K$ is convex) for all $N \in \mathcal{N}$. Furthermore, the graph completion of $\mu_N$ is defined by $\theta_N:[0,1]\rightarrow [0,T]$ and $\phi_N:[0,1]\rightarrow K$ as
$$
\theta_N(s):=\bar{t}_k+\frac{s-s_k}{h}(\bar{t}_{k+1}-\bar{t}_k) \; \mbox{whenever} \; s \in [s_k,s_{k+1}],
$$
$$
\phi_N(s):=F_N\circ\theta_N(s),
$$
where $h=1/N$, $s_k=kh$ and $k=0,...,N-1$, and it satisfies:\\
i) There exists a constant $b>0$ so that $\theta_N(\cdot)$ is Lipschitz of rank $b$ for all $N \in \mathcal{N}$;\\
ii) There exists a constant $r>0$ so that $\lim\sup_{N\rightarrow \infty}\|\dot{\phi}_N(\cdot)\|_{\infty}\leq r$.

Note that $\cup_{N \in \mathcal{N}} \mathcal{C}_N$ is dense in $\mathbb{R}^n$ and $\cup_{N \in \mathcal{N}}L_{N}^m$ is dense in $L_{\infty,2}^m[0,1]$.

Now, define
$$
\tilde{S}_{C,N}:= \tilde{S}_C \cap S_N.
$$

We get some results.

\begin{lemma}\label{lem}
$\cup \mathcal{P}_N$ is dense in $\mathcal{P}$ (endowed with the metric $d_3$).
\end{lemma}

\begin{proof}
For the first inclusion, let $\Omega\in \overline{\cup\mathcal{P}_N}$, then there exists a sequence $\{\Omega_N\}_{N \in \mathcal{N}}$ $\subset \cup\mathcal{P}_N$ so that $\Omega_N\rightarrow^{d_3}\Omega$, that is, $d_4(\Omega_N,\Omega)\rightarrow 0$ as $N\rightarrow \infty$. As $\mathcal{P}$ is the completion of the absolutely measures in the metric $d_4$, $\Omega \in \mathcal{P}$. Thence,
$$
\overline{\cup\mathcal{P}_N} \subseteq \mathcal{P}.
$$

Let $\Omega:=(\mu,|\mu|,\psi_{t_i}) \in \mathcal{P}$. We need to show there exists a sequence of $\cup \mathcal{P}_N$ converging to $\Omega$ in the metric $d_3$. Let $(\theta,\phi)$ be the gaph completion of $\mu$. By \cite{Peter}, there exists a partition of $[0,T]$, $0=:\bar{t}_0<\bar{t}_1<...<\bar{t}_N:=T$, and functions $\theta_N:[0,1]\rightarrow[0,T]$, $F_N:[0,T]\rightarrow\R^n$, $\phi_N:[0,1]\rightarrow K$, given by
$$
\theta_N(s)=\bar{t}_k+\frac{s-s_k}{h}(\bar{t}_{k+1}-\bar{t}_k) \;\; \mbox{whenever} \; s \in [s_k,s_{k+1}],
$$
where $h=1/N$ and $s_k=kh$,
$$
F_N(t;\mu_N)=\phi(s_k)+\frac{t-\bar{t}_k}{\bar{t}_{k+1}-\bar{t}_{k}}(\phi(s_{k+1})-\phi(s_k)) \;\; \mbox{whenever} \; t \in [\bar{t}_k,\bar{t}_{k+1}],
$$
$$
\phi_N(s)=(F_N\circ \theta_N)(s),
$$
and a measure given by
$$
d\mu_N=F_N(t;\mu_N)dt.
$$
Note that $\theta_N(\cdot)$ and $\phi_N(\cdot)$ are Lipschitz of rank $b$ and $r$, respectively, the same rank of the functions $\theta(\cdot)$ and $\phi(\cdot)$, respectively. These functions satisfy the graph-convergence,
$$
\int_0^1|\dot{\theta}_N(s)-\dot{\theta}(s)|ds\rightarrow 0 \;\;\;\mbox{and}\;\;\; \int_0^1|\dot{\phi}_N(s)-\dot{\phi}(s)|ds\rightarrow 0.
$$
We have the inequality
$$
0\leq |\phi_N(s)-\phi(s)|\leq \left|\int_0^1(\dot{\phi}_N(\tau)-\dot{\phi}(\tau))d\tau\right|+|\phi_N(0)-\phi(0)|\leq \int_0^1|\dot{\phi}_N(\tau)-\dot{\phi}(\tau)|d\tau.
$$
We can pass the limit in the last inequality and use the graph-convergence and the fact that $\phi_N(0)=\phi(0)$ to get
$$
\max_{s \in [0,1]}|\phi_N(s)-\phi(s)|\rightarrow 0.
$$
By \cite{Peter}, the graph-convergence is stronger than the weak$^{*}$ convergence, so $\mu_N\rightarrow^{*}\mu$. By Banach-Steinhaus theorem, \cite{MR1157815}, as $\mu_N\rightarrow^{*}\mu$, there exists $c>0$ so that $\|\mu_N\|\leq c$ for all $N$, where $\|\mu_N\|$ is the total variation of the measure $\mu_N$. By Helly's Theorem, \cite{Markov}, it is possible to construct a measure $\nu$ on $[0,T]$ and select from $|\mu_N|$ a subsequence such that $|\mu_N|\rightarrow^{*}\nu$. As $(\mu_N,|\mu_N|)\rightarrow^{*}(\mu,\nu)$ and our measures are positives we conclude that $\nu=|\mu|$. By Lemma 7.1, page 134, \cite{MR2161606}, $F_N(t;\mu_N)\rightarrow F(t;\mu)$ for all $t \in$ Cont$|\mu|$, where Cont$|\mu|$ denotes all points of continuity of the scalar-valued measure $|\mu|$. As the set of all points of discontinuity of $|\mu|$ has null Borel measure we can conclude that $F_N(t;\mu_N)\rightarrow F(t;\mu)$ a.e. $t \in [0,T]$.

Note that for $t \in [0,T]$, there exists $k \in \{0,...,N-1\}$ so that $t \in [\bar{t}_k,\bar{t}_{k+1}]$, and
$$
|F_N(t;\mu_N)|=\left|\phi(s_k)+\frac{t-\bar{t}_k}{\bar{t}_{k+1}-\bar{t}_k}(\phi(s_{k+1})-\phi(s_k))\right|\leq \left|\phi(s_k)\right|+
$$
$$
+\left|\frac{t-\bar{t}_k}{\bar{t}_{k+1}-\bar{t}_k}\right|\left|(\phi(s_{k+1})-\phi(s_k))\right| \leq \left|\phi(s_k)\right|+\left|\phi(s_{k+1})-\phi(s_k)\right|.
$$  
As $\phi(\cdot)$ is continuous and defined over the compact set $[0,1]$, there exists $M>0$ so that $|\phi(s)|\leq M$ for all $s \in [0,1]$. So, $|F_N(t;\mu_N)|\leq 3M$ for all $N \in \mathcal{N}$ and $t \in [0,T]$. As $M$ is integrable and $F_N(\cdot;\mu_N)$ is absolutely continuous, then $F_N(t;\mu_N)$ is integrable for each $N \in \mathcal{N}$. By the Dominate Convergence Theorem,
$$
\int_0^T|F_N(t;\mu_N)-F(t;\mu)|dt \rightarrow 0,\; N\rightarrow \infty.
$$
Note that $\nu_N:=|\mu_N|$ is a measure from $[0,T]$ to $\R_{+}$ and $|\nu_N|=\nu_N$. Then we have $(\nu_N,|\nu_N|)\rightarrow^{*}(|\mu|,|\mu|)$. Again, by Lemma 7.1, page 134, \cite{MR2161606}, $F_N(t;\nu_N):=\nu_N([0,t])\rightarrow F(t;|\mu|):=|\mu|([0,t])$ for all $t \in$ Cont$(|\mu|)$. As above, $F_N(t;\nu_N)\rightarrow F(t;|\mu|)$ a.e. $t \in [0,T]$. As $\nu_N$ is increasing we must have $\nu_N([0,t])\leq \|\nu_N\|=\|\mu_N\|\leq c$, that is, $|F_N(t;\nu_N)|\leq c$ for all $N \in \mathcal{N}$ and $t \in [0,T]$. We can use the same argument above and get
$$
\int_0^T|F_N(t;\nu_N)-F(t;|\mu|)|dt \rightarrow 0, \; N\rightarrow \infty.
$$
Then,
\[\begin{array}{lll}
\vspace{0.3cm}
0\leq \int_0^T|F_N(t;(\mu_N,\nu_N))-F(t;(\mu,|\mu|))|dt \\
\;\;\;\; \leq\int_0^T|F_N(t;\mu_N)-F(t;\mu)|dt +\int_0^T|F_N(t;\nu)-F(t;|\mu|)|dt \rightarrow 0, \; N\rightarrow\infty.
\end{array}\]

By \cite{Vinter}, $[0,T]$ is a continuity set for any positive measure defined on $[0,T]$, and
$$
\int_{[0,T]}d\mu=\lim_{N\rightarrow\infty}\int_{[0,T]}d\mu_N \implies |\mu_N([0,T])-\mu([0,T])|\rightarrow 0.
$$
The same holds for $\nu_N$. Then, we get
$$
|(\mu_N,\nu_N)([0,T])-(\mu,|\mu|)([0,T])|\rightarrow 0.
$$
\end{proof}

By the density of the union of each set, follows that $\cup S_N$ is dense in $B$.

\begin{lemma}\label{teo1}
$\tilde{S}_{C,N} \rightarrow^{\mathcal{N}} \tilde{S}_C$, $N \rightarrow \infty$, where the convergence is in the sense of Kuratowski, \cite{Vinter}.
\end{lemma}

\begin{proof}
Let $\{\eta_N=(\xi^0_N,u_N,\Omega_N)\}_{N \in \mathcal{N}}$ be a sequence in $\tilde{S}_{C,N}$ such that $\eta_{N} \rightarrow^{d} \eta=(\xi^0,u,\Omega)$. As $\mathcal{C}$ is closed, $\xi^0 \in \mathcal{C}$. The part that $u \in \mathcal{U}_C$ is given by Proposition 4.3.1, \cite{Polak}. Now, we know that $\Omega_N=(\mu_N,|\mu_N|,0) \rightarrow^{d_3}\Omega=(\mu,|\mu|,\{\psi_{t_i}\})$. As it was mentioned, $\mathcal{P}$ is the completion of the set of absolutely continuous vector-valued measures given on $[0,T]$ in the metric $d_4$. As the other part of the metric $d_3$ ($d_5$) is only completing the other one we can conclude that $\Omega \in \mathcal{P}$.
$$
\therefore \overline{\lim}\tilde{S}_{C,N} \subset \tilde{S}_C.
$$
Now, take $\eta=(\xi^0,u,\Omega) \in \tilde{S}_C$. We must find a sequence in $\tilde{S}_{C,N}$ that converges to $\eta$ in the metric $d$.
By Proposition 4.3.1, \cite{Polak}, and by Lemma \ref{lem}, it follows that there exists such a sequence.
$$
\therefore \tilde{S}_C \subset \underline{\lim}\tilde{S}_{C,N}.
$$
\end{proof}

Given $\eta=(\xi^0_N,u_N,\Omega_N) \in S_N$, we can use the Euler's descretization to get the discrete dynamic below by the continuous dynamic given by \eqref{defy.}. In this way, take $N \in \mathcal{N}$, $h=1/N$ the step size and $s_k=kh$, $k=0,...,N$. We have
\begin{equation}\label{discy}
\begin{array}{lll}
\vspace{0.2cm}
y^{\eta}_N(s_{k+1})-y^{\eta}_N(s_k)=f\left(y^{\eta}_N(s_k),u_N(s_k)\right)\left(\theta_N(s_{k+1})-\theta_N(s_k)\right)\\
\hspace{3.5cm}
\vspace{0.2cm}
+ g\left(y^{\eta}_N(s_k)\right)\left(\phi_N(s_{k+1})-\phi_N(s_k)\right),\\
 k=0,...,N-1,\; y^{\eta}_N(0)=\xi^0_N,
\end{array}
\end{equation}
where $\theta_N:[0,1]\rightarrow [0,T]$ and $\phi_N:[0,1]\rightarrow K$ are as defined in $\mathcal{P}_N$.

We associate the function
\begin{equation}\label{polig}
y^{\eta}_N(s):=\sum_{k=0}^{N-1}\left[y^{\eta}_N(s_k)+\frac{s-s_k}{h}\left(y^{\eta}_N(s_{k+1})-y^{\eta}_N(s_k)\right)\right]\tau_{N,k}(s),
\end{equation}
where $\{y^{\eta}_N(s_k)\}_{k=0}^{N}$ is the solution of the discrete system \eqref{discy}.

\begin{lemma}\label{lim}
Let $\eta=(\xi^0_N,u_N,\Omega_N) \in S_{N}$ and $\{y^{\eta}_N(s_k)\}_{k=0}^{N}$ be the solution of the discretized equation corresponding to this $\eta$. Thus, the following inequality holds
$$
|y^{\eta}_N(s_k)|+1 \leq e^{\beta}(1+|\xi^0_N|),
$$
where $\beta:=K_1(b+r)$, $K_1$ is the constant relative to the linear growth of the functions $f(\cdot,\cdot)$ and $g(\cdot)$ and, $b$ and $r$ are the Lipschitz constants of the functions $\theta_N(\cdot)$ and $\phi_N(\cdot)$, respectively.
\end{lemma}

\begin{proof}
This result follows from the Corollary of the Discrete Gronwall's Lemma, \cite{Z:2005}.
\end{proof}

Define 
$$
S_{C,N}:=\{\eta \in \tilde{S}_{C,N}:|y^{\eta}_N(s)|\leq L+1/N \;\;\; \forall s \in [0,1]\},
$$
where $y^{\eta}_N(\cdot)$ is given by \eqref{polig}.

We need to show that $S_{C,N}$ is converging to $S_C$, and, after that, we can define the approximated problems.

\begin{theorem}\label{SCN}
$S_{C,N} \rightarrow^{\mathcal{N}} S_C$, $N \rightarrow \infty$, where the convergence is in the sense of Kuratowski.
\end{theorem}

\begin{proof}
Let $\{\eta_N=(\xi^0_N,u_N,\Omega_N)\}_{N \in \mathcal{N}}$ be a sequence in $S_{C,N}$ such that $\eta_N \rightarrow^d \eta=(\xi^0,u,\Omega)$. As $S_{C,N} \subset \tilde{S}_{C,N}$, by Lemma \ref{teo1}, $\eta \in \tilde{S}_C$. We know that $|y^{\eta}_N(s)|\leq L+1/N$ for all $ s \in [0,1]$. By Theorem \ref{teoimp}, there exists $\mathcal{K} \subset \mathcal{N}$ so that $y^{\eta_N}_N(\cdot)$ uniformly converges to $y^{\eta}(\cdot)$ in $\mathcal{K}$, then, given $\varepsilon=1/N$, there exists $N_0 \in \mathcal{N}$ such that for all $N \geq N_0$, $N \in\mathcal{K}$ we have 
$$
|y^{\eta_N}_N(s)-y^{\eta}(s)|\leq 1/N \implies |y^{\eta}(s)|\leq |y^{\eta_N}_N(s)|+1/N \rightarrow L.
$$
$$
\therefore \overline{\lim}S_{C,N} \subset S_C.
$$
Now, let $\eta=(\xi^0,u,\Omega) \in S_C$. By Lemma \ref{teo1}, there exists a sequence $\{\eta_N=(\xi^0_N,u_N,\Omega_N)\}_{N \in \mathcal{N}} \in \tilde{S}_{C,N}$ so that $\eta_N\rightarrow^d \eta$. Again, by Theorem \ref{teoimp} there exists $\mathcal{K}\subset \mathcal{N}$ so that $y^{\eta_N}_N(\cdot)$ uniformly converges to $y^{\eta}(\cdot)$ in $\mathcal{K}$, then, given $\varepsilon=1/N$ there exists $N_0 \in \mathcal{N}$ such that for all $N \geq N_0$, $N \in\mathcal{K}$ we have 
$$
|y^{\eta_N}_N(s)-y^{\eta}(s)|\leq 1/N \implies |y^{\eta_N}_N(s)|\leq |y^{\eta}(s)|+1/N\leq L+1/N.
$$
$$
\therefore S_C\subset \underline{\lim}S_{C,N}.
$$
\end{proof}

Then, we get the approximated problems
$$
(P_{rep}^{C,N}) \;\;\; \min_{\eta \in S_{C,N}}f^0_N(y^{\eta}_N(0),y^{\eta}_N(1)),
$$
where $f^0_N(y^{\eta}_N(0),y^{\eta}_N(1)):=f^0(\xi^0_N,y^{\eta}_N(1))$.

\section{Consistent Approximations}\label{3.2}

In this section, we show that the problems $(P_{rep}^{C,N})$ with some optimality functions $\gamma^{C,N}(\cdot)$ are consistent approximations to the pair $(P_{rep},\gamma)$, where $\gamma(\cdot)$ is an optimality function to the problem $(P_{rep})$.

We begin with a theorem that gives us a convergence between the polygonal arc given by Euler's discretization and the solution of the reparametrized problem. Note that $\eta_N \in \tilde{S}_{C,N}$ is arbitrary and is converging to $\eta \in \tilde{S}_C$ in the metric $d$. This means that the convergence in the metric $d$ is enough to guarantee the convergence between the solutions.

\begin{theorem}\label{teoimp}
Suppose that the Assumption \ref{assump} holds, $N \in \mathcal{N}$ and $\eta_N \rightarrow^{d} \eta$, where $\eta_N \in \tilde{S}_{C,N}$ and $\eta \in \tilde{S}_C$. Thus, there exists $\mathcal{K} \subset \mathcal{N}$ such that $y^{\eta_N}_N(\cdot)$ uniformly converge to $y^{\eta}(\cdot)$, with $N \in \mathcal{K}$, $N \rightarrow \infty$, where $y^{\eta_N}_N(\cdot)$ is defined in \eqref{polig} and $y^{\eta}(\cdot)$ is the solution of \eqref{defy.}.
\end{theorem}

\begin{proof}
Note that, over the interval $[s_k,s_{k+1}]$ we have
\begin{equation}\label{equ}
\vspace{0.3cm}
|\dot{y}^{\eta_N}_N(s)| \leq K_1(b+r)|y^{\eta_N}_N(s_k)|+ K_1(b+r) =: \beta|y^{\eta_N}_N(s_k)| + \beta,
\end{equation}
and by Lemma \ref{lim},
$$
|y^{\eta_N}_N(s_{k})| \leq e^{\beta}\left(|\xi^0_N|+1\right)-1, \; k \in \{0, ..., N-1\}.
$$
As $\xi^0_N$ is convergent, it follows that there exists $\overline{M}>0$ such that $|\xi^0_N| \leq \overline{M} \;\; \forall N \in \mathcal{N}$. So, $|y^{\eta_N}_N(s_{k})|\leq e^{\beta}M$. By equation \eqref{equ}, $\dot{y}^{\eta_N}_N(s)$ is uniformly bounded. Using the same argument we have that $y^{\eta_N}_N(s)$ is uniformly bounded too. By Arzel$\grave{a}$-Ascoli's Theorem, there exist $\mathcal{K} \subset \mathcal{N}$ and $y:[0,1] \rightarrow \mathbb{R}^n$ such that $y^{\eta_N}_N(\cdot)$ uniformly converge to $y(\cdot)$, $N \in \mathcal{K}$, $N \rightarrow \infty$. \\
Now, we need to show that $y(\cdot)$ satisfies the system \eqref{defy.}. For this, define
$$
y^{\eta}(s)=f(y(s),u(s))\dot{\theta}(s) + g(y(s))\dot{\phi}(s), \;\; y^{\eta}(0)=\xi^0.
$$
For $s \in [s_k,s_{k+1}]$,
\[\begin{array}{lllll}
\vspace{0.5cm}
|y^{\eta_N}_N(s)-y^{\eta}(s)| \leq \left|\int_0^{s}\left(\dot{y}^{\eta_N}_N(\sigma)-\dot{y}^{\eta}(\sigma)\right)d\sigma\right| +|\xi^0_N-\xi^0| \\
\vspace{0.5cm}
\quad \; \leq \sum_{j=0}^{k-1}\int_{s_j}^{s_{j+1}}\left|\dot{y}^{\eta_N}_N(\sigma)-\dot{y}^{\eta}(\sigma)\right|d\sigma +\int_{s_k}^s\left|\dot{y}^{\eta_N}_N(\sigma)-\dot{y}^{\eta}(\sigma)\right|d\sigma +|\xi^{0}_N-\xi^{0}|\\
\vspace{0.5cm}
\quad \; \leq \sum_{j=0}^{k-1}\int_{s_j}^{s_{j+1}}|f(y^{\eta_N}_N(s_j),u_N(s_j))-f(y(\sigma),u(\sigma))||\dot{\theta}_N(\sigma)|d\sigma \\
\vspace{0.5cm}
\quad \; + \sum_{j=0}^{k-1}\left[\int_{s_j}^{s_{j+1}}|f(y(\sigma),u(\sigma))|\left|\dot{\theta}_N(\sigma)-\dot{\theta}(\sigma)\right|d\sigma+\int_{s_j}^{s_{j+1}}|g(y(\sigma))|\left|\dot{\phi}_N(\sigma)-\dot{\phi}(\sigma)\right|d\sigma\right]\\
\vspace{0.5cm}
\quad \; + \sum_{j=0}^{k-1}\int_{s_j}^{s_{j+1}}|g(y^{\eta_N}_N(s_j))-g(y(\sigma))||\dot{\phi}_N(\sigma)|d\sigma \\
\vspace{0.5cm}
\quad \; +\int_{s_k}^{s}|\dot{y}^{\eta_N}_N(\sigma)-\dot{y}^{\eta}(\sigma)|d\sigma +|\xi^0_N-\xi^0|\\
\vspace{0.5cm}
\quad \; =\sum_{j=0}^{k-1}[I+II+III+IV]+\int_{s_k}^{s}|\dot{y}^{\eta_N}_N(\sigma)-\dot{y}^{\eta}(\sigma)|d\sigma+V.
\end{array}\]
Let's check that there exists $\mathcal{K} \subset \mathcal{N}$ such that the equation above converge to zero whenever $N \in \mathcal{K}$.\\
- For $I$, as $f(\cdot,\cdot)$ is Lipschitz,
\[\begin{array}{lllll}
\vspace{0.5cm}
I\leq \int_{s_j}^{s_{j+1}}K^{'}b\left(|y^{\eta_N}_N(s_j)-y(\sigma)|+|u_N(s_j)-u(\sigma)|\right)d\sigma \leq bK^{'}\int_{s_j}^{s_{j+1}}\left(|y^{\eta_N}_N(s_j)-y^{\eta_N}_N(\sigma)|\right)d\sigma\\
\vspace{0.5cm}
\qquad \qquad \quad +bK^{'}\int_{s_j}^{s_{j+1}}\left(|y^{\eta_N}_N(\sigma)-y(\sigma)| + |u_N(s_j)-u_N(\sigma)| + |u_N(\sigma)-u(\sigma)|\right)d\sigma,\\
\end{array}\]
since $\sup|\dot{\theta}_N(s)|\leq b$, for some $b>0$.\\
It is easy to verify that $y^{\eta_N}_N(\cdot)$ is Lipschitz. Let's denote its Lipschitz constant by $\kappa>0$. Then,
$$
\int_{s_j}^{s_{j+1}}|y^{\eta_N}_N(s_j)-y^{\eta_N}_N(\sigma)|d\sigma \leq \int_{s_j}^{s_{j+1}}\kappa|s_j-\sigma|d\sigma \leq \kappa h^2\rightarrow 0. 
$$
As $y^{\eta_N}_N(\cdot)$ uniformly converge to $y(\cdot)$ we have that $|y^{\eta_N}_N(\sigma)-y(\sigma)|\rightarrow 0$. As every sequence uniformly convergent is bounded, there exists $c>0$ so that
$$
|y^{\eta_N}_N(\sigma)-y(\sigma)|\leq c \;\; \forall N,\; \forall \sigma \in [0,1],
$$
then,
$$
\int_{s_j}^{s_{j+1}}|y^{\eta_N}_N(\sigma)-y(\sigma)|d\sigma \leq ch\rightarrow 0.
$$
As $u_N(s)=u_N(s_j)$ for all $s \in [s_j,s_{j+1}[$ we get
$$
\int_{s_j}^{s_{j+1}}|u_N(s_j)-u_N(\sigma)|d\sigma \rightarrow 0.
$$
We know $u_N(\sigma)\rightarrow^{d_2} u(\sigma)$. By H$\ddot{o}$lder's inequality, we get $u_N(\sigma)\rightarrow u(\sigma)$ in $L_1^m([0,1])$.

- For $II$, as $y^{\eta_N}_N(\cdot)$ uniformly converge to $y(\cdot)$, given $\varepsilon=1$, there exists $N_0 \in \mathbb{N}$ so that for all $N\geq N_0$, $|y(s)-y^{\eta_N}_N(s)|<1$ for all $s \in [0,1]$, i.e., $|y(s)|<1+|y^{\eta_N}_N(s)|<\hat{M}$, for some $\hat{M}>0$ since $y^{\eta_N}_N(\cdot)$ is uniformly bounded. As $f$ has linear growth,
\[\begin{array}{lll}
|f(y(\sigma),u(\sigma))| \leq K_1(1+|y(s)|) \leq K_1(1+\hat{M}).
\end{array}\]
By the convergence of $\eta_N$ in the metric $d$, we have
$$
0\leq \int_{s_j}^{s_{j+1}}|\dot{\theta}_N(s)-\dot{\theta}(s)|ds\leq \int_0^1|\dot{\theta}_N(s)-\dot{\theta}(s)|ds \rightarrow 0.
$$

- $III$ and $IV$ are completely analogous to $II$ and $I$, respectively.

- As $\eta_N \rightarrow^{d} \eta$, then there exists the convergence between the initials conditions, then we have the convergence of $V$. The last integral is totally analogous to the one that we just showed.
\end{proof}

In the same way that we did in the last theorem, we can prove the same result when the sequence of $\eta's$ belongs to the set $S_C$ and also the point of convergence and when both of them belong to the set $S_{C,N}$.

\begin{prop}\label{solcont}
a) Let $\{\eta_N=(\xi^0_N,u_N,\Omega_N)\}_{N \in \mathcal{N}} \subset S_{C}$ be a sequence so that $\eta_N \rightarrow^d \eta$, where $\eta \in S_C$. Thus, there exists $\mathcal{K} \subseteq \mathcal{N}$ such that $y^{\eta_N}(\cdot)$ uniformly converge to $y^{\eta}(\cdot)$ when $N \rightarrow \infty$, $N \in \mathcal{K}$, where $y^{\eta_N}(\cdot)$ and $y^{\eta}(\cdot)$ are the solution of the system \eqref{defy.} related to $\eta_N$ and $\eta$, respectively.\\
b) Let $\{\eta_N=(\xi^0_N,u_N,\Omega_N)\}_{N \in \mathcal{N}} \subset S_{C,N}$ be so that $\eta_N \rightarrow^d \eta$, $\eta \in S_{C,N}$. Let $y^{\eta_N}_N(\cdot)$ and $y^{\eta}(\cdot)$ be the polygonal arc given by the Euler's discretization, equation \eqref{polig}. Thus, there exists $\mathcal{K} \subseteq \mathcal{N}$ such that $y^{\eta_N}_N(\cdot)$ uniformly converge to $y^{\eta}(\cdot)$ when $N \rightarrow \infty$, $N \in \mathcal{K}$.
\end{prop}

\begin{obs}\label{obs5}
Note that Theorem \ref{teoimp} holds when we change $\tilde{S}_C$ and $\tilde{S}_{C,N}$ by $S_C$ and $S_{C,N}$, respectively. The proof follows from Theorem \ref{teoimp} and the proof of Theorem \ref{SCN}.
\end{obs}

The following Lemma is very important to the next result.

\begin{lemma}\label{teodef}
Let $\varphi:S_C \rightarrow \mathbb{R}$ be given by
$$
\varphi(\bar{\eta})=\langle\nabla f^0(\xi),\bar{\xi}-\xi\rangle +\frac{1}{2}\bar{d}((\xi^0,u,\Omega),(\bar{\xi}^0,\bar{u},\bar{\Omega}))
$$
for each $\eta \in S_C$ fixed and $\xi=(\xi^0,y^{\eta}(1)) \in \mathcal{C}\times \mathbb{R}^n$. Then there exists $\hat{\eta} \in S_C$ such that
$$
\varphi(\hat{\eta})=\min_{\bar{\eta}\in S_C}\varphi(\bar{\eta}).
$$
\end{lemma}

\begin{proof}
Let $\alpha=\inf_{\hat{\eta} \in S_C}\varphi(\hat{\eta})$, then by the definition of infimum, there exists $\alpha_N=\varphi(\eta_N)$ so that $\alpha_N \rightarrow \alpha$ (in $\R$), and $\eta_N \in S_C$ for all $N \in \mathcal{N}$. As $\alpha_N$ is a convergent sequence in $\R$, it must exists $M>0$ so that $|\alpha_N|\leq M$ for all $N \in \mathcal{N}$, that is
$$
\langle\nabla f^0(\xi),\xi_N-\xi\rangle +\frac{1}{2}\bar{d}((\xi^0,u,\Omega),(\xi^0_N,u_N,\Omega_N))\leq M.
$$
As $\eta_N \in S_C$ for all $N \in \mathcal{N}$, we must have $\sup_{s \in [0,1]}|y^{\eta_N}(s)|\leq L$ for all $N \in \mathcal{N}$, then $\xi_N$ is uniformly bounded, that is, $\langle\nabla f^0(\xi),\xi_N-\xi\rangle$ is uniformly bounded. We have
$$
d_1(\xi^0,\xi^0_N)\leq M_1, \;\;\; d_2(u_N,u)\leq M_1 \;\; \mbox{and} \;\; d_4(\Omega_N,\Omega)\leq M_1,
$$
for some $M_1>0$.

We have some points:
\begin{itemize}
\item $d_1(\xi^0,\xi^0_N)\leq M_1 \implies |\xi^0_N|\leq M_2$, $M_2$ is some positive constant. Then there exist $\mathcal{K}_1 \subset \mathcal{N}$ and $\bar{\xi}^0 \in \mathcal{C}$ so that $\xi^0_N\rightarrow \bar{\xi}^0$;\\
\item $d_2(u_N,u)\leq M_1\implies \int_0^1|u_N(s)|^2 \leq M_3$ if we use Minkoviski's inequality, $M_3$ is some positive constant. By the sequential compactness in $L_2^m([0,1])$, there exist $\mathcal{K}_2 \subset \mathcal{K}_1$ and $\bar{u}\in L_2^m([0,1])$ so that
$$
\int_0^1\langle u_N(s)-\bar{u}(s),h(s)\rangle ds \rightarrow 0 \;\;\; \forall h \in L_2^m([0,1]), N \in \mathcal{K}_2.
$$
\end{itemize}
We need to show that $\bar{u} \in \mathcal{U}_C$. We know $u_N(s) \in \overline{\mathcal{U}}$ a.e., for all $N \in \mathcal{K}_2$. Define
$$
W:=\{\omega\in L_2^m([0,1]):\omega(t)\in \overline{\mathcal{U}} \; \mbox{a.e.} t \in [0,1]\}.
$$
Then $W$ is strongly closed in $L_2^m([0,1])$, because a strongly convergent sequence admits a subsequence converging almost everywhere, and $\overline{\mathcal{U}}$ is closed by assumption. $W$ is convex because $\overline{\mathcal{U}}$ is convex. By Theorem III.7, \cite{Brezis}, $W$ is weakly closed. As $\bar{u}$ is the weak limit of the sequence $u_N$, $\bar{u}$ belongs to $W$, and then $\bar{u} \in \mathcal{U}_C$.

\begin{itemize}
\item $d_4(\Omega_N,\Omega)\leq M_1$.\\
By a statement given by \cite{K:2005}, page 7105, there exist $\mathcal{K}_3 \subset \mathcal{K}_2$ and $\bar{\Omega} \in \mathcal{P}$ so that $d_4(\Omega_N,\bar{\Omega})\rightarrow 0, N \in \mathcal{K}_3$.
\end{itemize}

Then, when $N \in \mathcal{K}_3$, we have\\
1) $\xi^0_N \rightarrow^{d_1}\bar{\xi}^0$;\\
2) $u_N \rightarrow \bar{u}$ weakly in $L_2^m([0,1])$;\\
3) $\Omega_N \rightarrow^{d_4}\bar{\Omega}$.

By theorem 6.1, \cite{K:2005}, if $f$ is linear in $u$, 1), 2) and 3) hold and $\sup_{s \in [0,1]}|y^{\eta_N}(s)| \leq L$, we can still apply Lemma 3.2, \cite{K:2005}, and get that $y^{\eta_N}(1) \rightarrow y^{\bar{\eta}}(1)$, where $y^{\bar{\eta}}(\cdot)$ is the trajectory of the reparametrized system related to $\bar{\eta}=(\bar{\xi}^0,\bar{u},\bar{\Omega})$. Moreover, $\sup_{s \in [0,1]}|y^{\eta_N}(s)|\rightarrow \sup_{s\in[0,1]}|y^{\bar{\eta}}(s)|$, and as $\sup_{s\in[0,1]}|y^{\eta_N}(s)|\leq L$, we must have that $\sup_{s\in[0,1]}|y^{\bar{\eta}}(s)|\leq L$. Then, $\bar{\eta} \in S_C$. 

We have strong convergence in $\mathcal{C}$ and $\mathcal{P}$ but we have weak convergence in $L_2^m([0,1])$. Let us work on that.

We know that $d_2:\mathcal{U}_C\rightarrow\R$  and $\mathcal{U}_C \subset L_{\infty,2}^m([0,1])\subset L_2^m([0,1])$, where $ L_2^m([0,1])$ is a Banach space. Let $\lambda$ be so that there exists $\hat{u} \in \mathcal{U}_C$ satisfying $d_2(u,\hat{u})=\lambda$. Define
$$
A:=\{\tilde{u}\in \mathcal{U}_C:d_2(u,\tilde{u})\leq \lambda\}.
$$
As $\mathcal{U}_C$ and $d_2$ are convex, $A$ is convex.

If we take a sequence $\{\tilde{u}_N\}_{N \in \mathcal{N}}$ in $A$ so that $\tilde{u}_N\rightarrow^{d_2} \tilde{\tilde{u}}$, we must have that $\tilde{u}_N \in \mathcal{U}_C$ for all $N \in \mathcal{N}$ and $d_2(\tilde{u}_N,u)\leq \lambda$. As $d_2$ is continuous we have that $d_2(\tilde{\tilde{u}},u)\leq \lambda$. We need to show that $\tilde{\tilde{u}} \in \mathcal{U}_C$. In the same way we showed $\bar{u} \in W$, we can get that $\tilde{\tilde{u}} \in \mathcal{U}_C$. Then, $A$ is strongly closed. By Theorem III.7, \cite{Brezis}, $A$ is weakly closed. In particular we have that if $u_N \rightarrow \bar{u}$ weakly in $L_2^m([0,1])$ then
$$
d_2(u,\bar{u})\leq \liminf_{N \rightarrow\infty} d_2(u_N,u).
$$
We can write
\[\begin{array}{lll}
\vspace{0.3cm}
\varphi(\bar{\eta})=\langle \nabla f^0(\xi),\bar{\xi}-\xi\rangle +d_1(\xi^0,\bar{\xi}^0)+d_2(u,\bar{u})+d_4(\Omega,\bar{\Omega})\\
\vspace{0.3cm}
\quad \;\quad \leq \lim_{N \rightarrow\infty}\left[\langle \nabla f^0(\xi),\xi_N-\xi\rangle +d_1(\xi^0_N,\xi^0)+d_4(\Omega_N,\Omega)\right] + \liminf_{N \rightarrow\infty} d_2(u_N,u)\\
\vspace{0.3cm}
\quad \;\quad =\liminf_{N \rightarrow\infty}\left[\langle \nabla f^0(\xi),\xi_N-\xi\rangle +d_1(\xi^0_N,\xi^0)+d_2(u_N,u)+d_4(\Omega_N,\Omega)\right]\\
\quad \;\quad =\liminf_{N\rightarrow\infty} \varphi(\eta_N)=\alpha.
\end{array}\]
Therefore, $\varphi$ achieves its minimum over $S_C$.
\end{proof}

The same result can be proved when we change the domain of $\varphi$ by $S_{C,N}$.

The next result provides the optimality functions to the problems $(P_{rep})$ and $(P^{C,N}_{rep})$.

\begin{theorem}\label{otimalidade}
Suppose that Assumption \ref{assump} holds. The following statements are satisfied:
a) Let
$$
\gamma(\eta):= \min_{\bar{\eta} \in S_C}\left(\langle\nabla f^0(\xi),\bar{\xi}-\xi\rangle +\frac{1}{2}\bar{d}((\xi^0,u,\Omega),(\bar{\xi}^0,\bar{u},\bar{\Omega}))\right),
$$
with $\xi:=(\xi^0,y^{\eta}(1))$, $\bar{\xi}:=(\bar{\xi}^0,y^{\bar{\eta}}(1))$, $\bar{d}=d_1+d_2+d_4$ and $\gamma:S_C \rightarrow \mathbb{R}$.\\
i) If $\bar{\eta} \in S_C$  is a local minimizer of $(P_{rep})$, then
$$
\langle \nabla f^0(\bar{\xi}),\xi-\bar{\xi}\rangle \geq 0 \;\;\; \forall \; \eta \in S_C;
$$
ii) $\gamma(\eta) \leq 0 \;\; \forall \; \eta \in S_C$;\\
iii) If $\bar{\eta} \in S_C$ is a minimum of $(P_{rep})$, then $\gamma(\bar{\eta})=0$;\\
iv) $\gamma(\cdot)$ is upper semicontinuous.\\
Thus, we can conclude that $\gamma$ is an optimality function to the problem $(P_{rep})$.

\noindent b) Let
$$
\gamma^{C,N}(\eta)=\min_{\bar{\eta} \in S_{C,N}}\left(\langle\nabla f^0_N(\xi),\bar{\xi}-\xi\rangle+\frac{1}{2}\bar{d}((\xi^0,u,\Omega),(\bar{\xi}^0,\bar{u},\bar{\Omega}))\right),
$$
with $\gamma^{C,N}:S_{C,N} \rightarrow \mathbb{R}$.\\
i) If $\bar{\eta} \in S_{C,N}$ is a local minimizer of $(P^{C,N}_{rep})$, then
$$
\langle \nabla f^0_N(\bar{\xi}),\xi-\bar{\xi}\rangle \geq 0 \;\;\; \forall \; \eta \in S_{C,N};
$$
ii) $\gamma^{C,N}(\eta) \leq 0 \;\; \forall \; \eta \in S_{C,N}$;\\
iii) If $\bar{\eta} \in S_{C,N}$ is minimum of $(P^{C,N}_{rep})$, then $\gamma^{C,N}(\bar{\eta})=0$;\\
iv) $\gamma^{C,N}(\cdot)$ is upper semicontinuous;\\
Thus, we can conclude that $\gamma^{C,N}$ is an optimality function to the problem $(P^{C,N}_{rep})$.
\end{theorem}

\begin{proof}
a). i) Suppose that $\bar{\eta} \in S_C$ is a minimizer of $(P_{rep})$ and there exists $\eta \in S_C$ so that
\begin{equation}\label{*}
\langle \nabla f^0(\bar{\xi}),\xi-\bar{\xi} \rangle <0.
\end{equation}
As $\mathcal{C}$ is convex, $\bar{\xi}+\lambda(\xi-\bar{\xi}) \in \;\mathcal{C} \times \mathbb{R}^n \;\; \forall \; \lambda \in [0,1]$. There exists $\bar{\lambda} \in (0,1]$ such that
$$
f^0(\bar{\xi}+\bar{\lambda}(\xi-\bar{\xi}))-f^0(\bar{\xi}) \leq \bar{\lambda}\langle \nabla f^0(\bar{\xi}),\xi-\bar{\xi}\rangle <0,
$$
where in the last inequality we used the inequality \eqref{*}. 

Let $\xi^0=\bar{\xi}^0+\bar{\lambda}(\xi^0-\bar{\xi}^0)$ and $\xi_1=y^{\bar{\eta}}(1)+\bar{\lambda}(y^{\eta}(1)-y^{\bar{\eta}}(1))$ be the fixed initial and final points of the reparametized system \eqref{defy.}. As such system is controllable it must exists a solution of the reparametrized system $y(\cdot)$ satisfying $y(0)=\xi_0$ and $y(1)=\xi_1$, i.e.,
$$
f^0(\bar{\xi}+\bar{\lambda}(\xi-\bar{\xi}))<f^0(\bar{\xi}).
$$ 
This is a contradiction with the assumption.

\noindent ii) Note that
$$
\gamma(\eta) \leq \left(\langle\nabla f^0(\xi),\xi-\xi\rangle +\frac{1}{2}\bar{d}((\xi^0,u,\Omega),(\xi^0,u,\Omega))\right)=0 \;\;\; \forall \; \eta \in S_C,
$$
because as $\gamma(\eta)$ is the minimum, it is smaller than when calculated in $\eta$.

\noindent iii) Suppose that $\bar{\eta} \in S_C$ is a minimizer of $(P_{rep})$. We have
$$
0 \geq \gamma(\bar{\eta}) \geq \min_{\eta \in S_C}\langle \nabla f^0(\bar{\xi}),\xi-\bar{\xi}\rangle\geq 0,
$$
where in the first and third inequality we used the items ii) and i), respectively.

\noindent iv) Let $\{\eta_N\}_{N \in \mathcal{N}}$ be a sequence of $S_C$ so that $\eta_N \rightarrow^d \eta, \;N \in \mathcal{N}$, and $\mathcal{K} \subset \mathcal{N}$ such that
$$
\overline{\lim}\gamma(\eta_N)=\lim_{N \in \mathcal{K}}\gamma(\eta_N).
$$
By Proposition \ref{solcont} part a), there exists $\bar{\mathcal{K}} \subset \mathcal{K}$ such that $y_{N}^{\eta_N}(1) \rightarrow y^{\eta}(1)$ whenever $N \in \bar{\mathcal{K}}$. Then, by the definition of $\gamma(\cdot)$ and Assumption \ref{assump}, it follows that $\lim_{N \in \bar{\mathcal{K}}}\gamma(\eta_N)=\gamma(\eta)$. But, as the limit $\lim_{N \in \mathcal{K}}\gamma(\eta_N)$ exists, we must have that all subsequence are converging to the same point, that is,
$$
\overline{\lim}\gamma(\eta_N)=\gamma(\eta).
$$

Part b) is totally analogous to the part a).
\end{proof}

\begin{theorem}\label{teo2}
Suppose that Assumption \ref{assump} holds. Then, $\{(P_{rep}^{C,N},\gamma^{C,N})\}_{N \in \mathcal{N}}$ is a sequence of consistent approximations to the pair $(P_{rep},\gamma)$.
\end{theorem}

\begin{proof}
To get this result, we need to show that the problems $(P_{rep}^{C,N})$ epi-converge to $(P_{rep})$ and $\overline{\lim}\gamma^{C,N}(\eta_N) \leq \gamma(\eta)$, when $\eta_N \rightarrow^d \eta$, $N \rightarrow \infty$, $N \in \mathcal{N}$.\\
\noindent Epi-convergence;\\
i) Let $\eta \in S_C$ be arbitrary. By the construction of $S_{C,N}$ itself and by the proof of Theorem \ref{SCN}, there exists $\{\eta_{N}\}_{N \in \mathcal{N}}$ such that $\eta_{N} \in S_{C,N}$, for all $N \in \mathcal{N}$, and $\eta_{N} \rightarrow^d \eta$, $N \rightarrow \infty$. Let $\mathcal{K} \subset \mathcal{N}$ be such that
$$
\overline{\lim}f^0_N(\xi^0_N,y^{\eta_N}_N(1))=\lim_{N \in \mathcal{K}}f^0_N(\xi^0_N,y^{\eta_N}_N(1)).
$$
By Observation \ref{obs5}, there exist $\mathcal{K}^{'} \subset \mathcal{K}$ such that $y^{\eta_N}_N(1) \rightarrow y^{\eta}(1)$, $N \in \mathcal{K}^{'}$. Then, we have
$$
\lim_{N \in \mathcal{K}^{'}}f^0_N(\xi^0_N,y^{\eta_N}_N(1))=f^0(\xi^0,y^{\eta}(1)),
$$
because of Assumption \ref{assump}. It follows that
$$
\lim_{N \in \mathcal{K}}f^0_N(\xi^0_N,y^{\eta_N}_N(1))=f^0(\xi^0,y^{\eta}(1)),
$$
because if the limit exist, all subsequences must converge to the same point. This gives us
$$
\overline{\lim}f^0_N(\xi^0_N,y^{\eta_N}_N(1))=f^0(\xi^0,y^{\eta}(1)).
$$

\noindent ii) Let $\{\eta_N\}_{N \in \mathcal{\mathcal{K}}}$ be a sequence such that $\eta_N \in S_{C,N}$ for all $N \in \mathcal{K}$ and $\eta_N \rightarrow^d \eta$, $N \rightarrow \infty$. By Theorem \ref{SCN}, we should have $\eta \in S_C$.
Take $\bar{\mathcal{K}} \subset \mathcal{K}$ such that
$$
\underline{\lim}_{N \in\mathcal{K}}f^0(\xi_N^0,y_N^{\eta_N}(1))=\lim_{N \in \bar{\mathcal{K}}}f^0(\xi_N^0,y^{\eta_N}_N(1)).
$$
As $\eta_N \rightarrow^{d} \eta$, $N \in \bar{\mathcal{K}}$, it follows from Observation \ref{obs5} that there exist $\bar{\bar{\mathcal{K}}} \subset \bar{\mathcal{K}}$ such that $y^{\eta_N}_N(1) \rightarrow y^{\eta}(1)$, $N \in \bar{\bar{\mathcal{K}}}$. By the same arguments used in the proof of the item i), it follows that $\underline{\lim}_{N \in\mathcal{K}}f^0(\xi_N^0,y_N^{\eta_N}(1))=f^0(\xi^0,y^{\eta}(1))$.
$$
\therefore P^{C,N}_{rep} \rightarrow^{Epi} P_{rep}.
$$

\noindent Now, let $\{\eta_N\}_{N \in \mathcal{N}}$ be a sequence such that $\eta_N \in S_{C,N}$ for all $N \in \mathcal{N}$ and it converges to $\eta \in S_C$. We must show that $\overline{\lim}\gamma^{C,N}(\eta_N) \leq \gamma(\eta)$, $N \rightarrow \infty$.
By Theorem Lemma \ref{teodef}, there exists $\bar{\eta} \in S_C$ such that
$$
\gamma(\eta)=\langle\nabla f^0(\xi^0,y^{\eta}(1)),(\bar{\xi}^0,y^{\bar{\eta}}(1))-(\xi^0,y^{\eta}(1))\rangle+\frac{1}{2}\bar{d}((\xi^0,u,\Omega),(\bar{\xi}^0,\bar{u},\bar{\Omega})).
$$
Let $\overline{\mathcal{K}} \subset \mathcal{N}$ be such that $\overline{\lim}\gamma^{C,N}(\eta_N)= \lim_{N \in \overline{\mathcal{K}}}\gamma^{C,N}(\eta_N)$. By Lemma \ref{teo1}, there exists $\{\bar{\eta}_N\}_{N \in \overline{\mathcal{K}}}$ so that $\bar{\eta}_N \rightarrow \bar{\eta}$, and $\bar{\eta}_N \in S_{C,N}$ for all $N \in \overline{\mathcal{K}}$. By the definition of $\gamma^{C,N}(\cdot)$,
\[\begin{array}{lll}
\vspace{0.3cm}
\gamma^{C,N}(\eta_N)\leq \langle \nabla f^0(\xi^0_N,y^{\eta_N}_N(1)),(\xi^0_N,y^{\eta_N}_N(1))-(\bar{\xi}^0_N,y^{\bar{\eta}_N}_N(1))\rangle\\
\qquad \qquad \qquad +\frac{1}{2}\bar{d}((\xi^0_N,u_N,\Omega_N),(\bar{\xi}^0_N,\bar{u}_N,\bar{\Omega}_N)).
\end{array}\]
By Observation \ref{obs5}, there exists $\mathcal{K} \subset \overline{\mathcal{K}}$ such that $y^{\eta_N}_N(1) \rightarrow y^{\eta}(1)$, and $y^{\bar{\eta}_N}_N(1) \rightarrow y^{\bar{\eta}}(1)$, $N \in \mathcal{K}$, then passing the limit with $N \rightarrow \infty$, $N \in \overline{\mathcal{K}}$, in the last inequality, we get
\[\begin{array}{lll}
\vspace{0.3cm}
\lim_{N \in \mathcal{K}}\gamma^{C,N}(\eta_N) \leq \langle\nabla f^0(\xi^0,y^{\eta}(1)),(\xi^0,y^{\eta}(1))-(\bar{\xi}^0,y^{\bar{\eta}}(1))\rangle\\
\quad \qquad \qquad \qquad\quad\quad +\frac{1}{2}\bar{d}((\xi^0,u,\Omega),(\bar{\xi}^0,\bar{u},\bar{\Omega}))=\gamma(\eta),
\end{array}\]
which give us
$$
\overline{\lim}\gamma^{C,N}(\eta_N) \leq \gamma(\eta).
$$
\end{proof}

The next Theorem shows that a local (respectively, global) minimizers sequence of $(P^{C,N}_{rep})$ that has a convergent subsequence is converging to a local (respectively, global) minimizer of $(P_{rep})$.

\begin{theorem}\label{teo8}
Let $(P^{C,N}_{rep})$ and $(P_{rep})$ be defined as before. Let $\{\eta_N\}_{N \in \mathcal{N}}$ be a sequence of local (respectively, global) minimizers of $(P^{C,N}_{rep})$ such that $\eta_N \rightarrow^{d} \eta$, with $N \rightarrow \infty$ and $\eta \in S_C$. Then $\eta$ is a local (respectively, global) minimizer of $(P_{rep})$ and there exists $\mathcal{K} \subset \mathcal{N}$ such that $f^0_N(\xi^0_{N},y^{{\eta}_N}_N(1)) \rightarrow f^0(\xi^0,y^{\eta}(1))$, with $N \rightarrow \infty$, $N \in \mathcal{K}$.
\end{theorem}

\begin{proof}
Let $\{\eta_N\}_{N \in \mathcal{N}}$ be a sequence of local minimizers of the problems  $(P^{C,N}_{rep})$, that is, there exists $\varepsilon >0$ so that for all $\hat{\eta} \in S_{C,N}$ satisfying $d(\eta_N,\hat{\eta})\leq \varepsilon$ we have $f^0_N(\xi^0_N,y^{\eta_N}_N(1))\leq f^0_N(\hat{\xi}^0,y^{\hat{\eta}}(1))$, such that $\eta_N \rightarrow^{d} \eta$. By Observation \ref{obs5} and Assumption \ref{assump}, there exists $\mathcal{K} \subset \mathcal{N}$ such that $y^{\eta_N}_N \rightarrow y^{\eta}$ uniformly and $f^0_N(\xi^0_N,y^{\eta_N}_N(1)) \rightarrow f^0(\xi^0,y^{\eta}(1))$, with $N \in \mathcal{K}$, $N \rightarrow \infty$.
We need to show that $\eta$ is a local minimizer to the problem $(P_{rep})$. Let's suppose that $\eta$ is not a local minimizer, then given $\varepsilon>0$ there exists $\bar{\eta} \in S_C$ with $d(\eta,\bar{\eta})<\varepsilon/3$ such that
$$
f^0(\bar{\xi}^0,y^{\bar{\eta}}(1))=f^0(\xi^0,y^{\eta}(1))-3\varepsilon,
$$
where $y^{\bar{\eta}}(\cdot)$ is the associated trajectory to $\bar{\eta}$.\\
As $(P^{C,N}_{rep})$ epi-converge to $(P_{rep})$ and $\eta_N \rightarrow^{d} \eta$ with $\eta_N \in S_{C,N}$, we have
\begin{equation}\label{eq10}
\underline{\lim}_{N \in \mathcal{K}}f_N^0(\xi^0_N,y^{\eta_N}_N(1)) \geq f^0(\xi^{0},y^{\eta}(1)).
\end{equation}
By the epi-convergence again, there exists a sequence $\{\bar{\eta}_N\}_{N \in \mathcal{N}}$ in $S_{C,N}$ such that $\bar{\eta}_N \rightarrow^{d} \bar{\eta}$ and
\begin{equation}\label{eq50}
\overline{\lim}_{N \in \mathcal{K}} f^0_N(\bar{\xi}^0_N,y^{\bar{\eta}_N}_N(1)) \leq \overline{\lim}f^0_N(\bar{\xi}^0_N,y^{\bar{\eta}_N}_N(1))\leq f^0(\bar{\xi}^0,y^{\bar{\eta}}(1)).
\end{equation}
Let $\mathcal{K}^{'} \subset \mathcal{K}$ be such that
$$
\underline{\lim}_{N \in \mathcal{K}}f_N^0(\xi_N^0,y^{\eta_N}(1))=\lim_{N \in \mathcal{K}^{'}}f_N^0(\xi^0_N,y^{\eta_N}(1)).
$$
As $\eta_N\rightarrow^d \eta$, given $\varepsilon/3$ there exists $N_1 \in \mathcal{K}^{'}$ so that $d(\eta_N,\eta)\leq \varepsilon/3$ for all $N \geq N_1$, $N \in \mathcal{K}^{'}$. As $\bar{\eta}_N\rightarrow^d \bar{\eta}$, given $\varepsilon/3$ there exists $N_2 \in \mathcal{K}^{'}$ so that $d(\bar{\eta}_N,\bar{\eta})<\varepsilon/3$ for all $N \geq N_2$, $N \in \mathcal{K}^{'}$. Let $N_3=\max\{N_1,N_2\}$ and $N\geq N_3$, $N \in \mathcal{K}^{'}$, then
$$
d(\bar{\eta}_N,\eta_N)\leq d(\eta_N,\eta)+d(\eta,\bar{\eta})+d(\bar{\eta},\bar{\eta}_N)\leq \varepsilon/3+\varepsilon/3+\varepsilon/3=\varepsilon,
$$
and there exists $N_4 \in \mathcal{K}^{'}$ so that for all $N \geq N_4$, $N \in \mathcal{K}^{'}$,
$$
\eqref{eq50} \Rightarrow f_N^0(\bar{\xi}_N^0,y^{\bar{\eta}_N}(1)) \leq f^0(\bar{\xi}^0,y^{\bar{\eta}}(1)) + \varepsilon = f^0(\xi^0,y^{\eta}(1)) -2\varepsilon.
$$
$$
\eqref{eq10} \Rightarrow f_N^0(\xi_N^0,y^{\eta_N}_N(1)) \geq f^0(\xi^0,y^{\eta}(1))-\varepsilon.
$$
It follows that
$$
f_N^0(\bar{\xi}_N^0,y^{\bar{\eta}_N}_N(1)) \leq f_N^0(\xi_N^0,y^{\eta_N}_N(1))-\varepsilon
$$
for all $N\geq N_0$, $N \in \mathcal{K}^{'}$, where $N_0=\max\{N_3,N_4\}$. Contradiction.
\end{proof}

Suppose $\{\eta_N\}_{N \in \mathbb{N}} \subset S_{C,N}$ is a sequence of global minimizers of $(P^{C,N}_{rep})$ that is converging to $\eta \in S_C$ in the metric $d$. By Theorem \ref{teo8}, $\eta$ is a global minimizer of $(P_{rep})$. Then $y^{\eta}(\cdot)$ given by \eqref{defy.} is the function that minimizes $(P_{rep})$.

Define $\pi:[0,T]\rightarrow [0,1]$ by
$$
\pi(t)=\frac{t+|\mu|([0,t])}{T+\|\mu\|},
$$
and $x:[0,T]\rightarrow \R^n$ by
$$
x(t)=y^{\eta}(\pi(t)).
$$
Because of Theorem \ref{teo9}, as $y^{\eta}(\cdot)$ is a solution of the system \eqref{defy.}, then $x(\cdot)$ is a solution of the original system \eqref{eq} related to $p=(\mu,|\mu|, \psi_{t_i})$ and $\bar{u}:[0,T] \rightarrow \R^m$ given by
$$
\bar{u}(t)=u(\pi(t)).
$$
As $y^{\eta}(\cdot)$ minimizes the reparametrized problem $(P_{rep})$ and $f^0(\xi^0,x(T))=f^0(\xi^0,y^{\eta}(1))$, we have that $x(\cdot)$ minimizes $(P)$.

Now, we can show that a subsequence of discrete-time approximated functions graph-converges to a solution. 

\begin{theorem}
Suppose $\eta_N \rightarrow^d \eta$ and define
$$
\Lambda^N:=\{(t_k,y_k):k=0,...,N\},
$$
and
$$
\mathbb{X}_{\mu}:=(x(\cdot),\phi(\cdot),\{\mathcal{X}_{t_i}\}_{t_i\in \Theta}),
$$
where $y_k=y_N^{\eta_N}(s_k)$, $t_k=\theta_N(s_k)$, $k=0,...,N$, $x(\cdot)$ is defined as above and
$$
\mbox{gr}\mathbb{X}_{\mu}:=\{(t,x(t)):t \in [0,T]\} \cup \{(t_i,y(s)):s \in I_i,i\in \mathcal{I}\}.
$$
Then, there exists $\mathcal{K}\subset \mathcal{N}$ so that
$$
\mbox{dist}_H(\Lambda^N,\mbox{gr}\mathbb{X}_{\mu})\rightarrow 0 \;\;\; \mbox{as} \;\; N\rightarrow \infty, N \in \mathcal{K},
$$
where the Hausdorff distance between two compact subsets $A,B \subset \R^m$ is given by
$$
\mbox{dist}_H(A,B)=\min\{\delta \geq 0: A\subseteq B+\delta B[0,1]\;\; \mbox{and} \;\; B\subseteq A+\delta B[0,1]\}.
$$ 
\end{theorem}

\begin{proof}
For each $N \in \mathcal{N}$, define
$$
\tilde{\Lambda}^N:=\{(s_k,y_k):k=0,...,N\}.
$$
Note that as $\eta_N \rightarrow^d \eta$, by Observation \ref{obs5}, there exists $\mathcal{K}\subset \mathcal{N}$ so that $y^{\eta_N}_N\rightarrow y^{\eta}$ uniformly when $N \in \mathcal{K}$, then we have
$$
\mbox{dist}_H(\mbox{gr}y^{\eta_N}_N,\mbox{gr}y^{\eta}) \rightarrow 0 \;\; \mbox{as} \; N \rightarrow \infty, N \in \mathcal{K}.
$$
Observe the second coordinates of $\Lambda^N$ and $\tilde{\Lambda}^N$ are equal for each $k=0,...,N$. In the same way, the second coordinates of gr$\mathbb{X}_{\mu}$ and gr$y$ are the same for each $t \notin \Theta$, and when $t \in \Theta$, the set of projections onto the second coordinate are the same. Then, we have the following:
$$
\mbox{dist}_H(\Lambda^N,\mbox{gr}\mathbb{X}_{\mu})\leq b\mbox{dist}_H(\tilde{\Lambda}^N,\mbox{gr}y),
$$
where $b$ is the Lipschitz constant of $\theta_N(\cdot)$.
We can get
\[\begin{array}{lll}
\mbox{dist}_H(\tilde{\Lambda}^N,\mbox{gr}y)\leq \mbox{dist}_H(\tilde{\Lambda}^N,\mbox{gr}y^{\eta_N}_N)+\mbox{dist}_H(\mbox{gr}y^{\eta_N}_N,\mbox{gr}y).
\end{array}\]
Passing the limit when $N \in \mathcal{K}$ in the last inequality we get the desired result.
\end{proof}

\section{Bounds on Approximation Errors}\label{section7}
In this section we give a bound on approximation errors between the linear interpolation of the sequence of Euler points and a trajectory of the reparametrized problem and also between the objective function of the approximated problem and the reparametrized problem.

\begin{obs}\label{obs1}
Note that Picard Lemma 5.6.3, \cite{Polak}, holds if we define $h:\R^n\times \R^m\times \R^{2q+1} \rightarrow \R^n$ as 
$$
h(x,u,\Omega)=f(x,u)\dot{\theta}(s)+g(x)\dot{\phi}(s) \;\; \mbox{a.e.} \; s \in [0,1],
$$
since $h(\cdot,\cdot,\cdot)$ is Lipschitz related to the first variable. Let $x,y \in \R^n$, $u \in \R^m$ and $\Omega \in \R^{2q+1}$, then
$$
|h(x,u,\Omega)-h(y,u,\Omega)|\leq |f(x,u)-f(y,u)||\dot{\theta}(s)|+||g(x)-g(y)||\dot{\phi}(s)| \leq (bK^{'}+rK^{''})|x-y|,
$$  
since $f$ and $g$ are Lipschitz in the first variable.
\end{obs}

\begin{theorem}
Suppose that Assumption 1 holds, $N \in \mathcal{N}$, and $S \subset B$ is a bounded set. Then there exists a constant $K_{\xi^0}>0$ such that, for any $\eta=(\xi^0,u,\Omega)$ and $\hat{\eta}=(\hat{\xi}^0,u,\Omega) \in S\cap S_N$, which differ only in the initial state,
$$
|y_N^{\hat{\eta}}(s)-y^{\eta}(s)|\leq K_{\xi^0}(|\xi^0-\hat{\xi}^0|+1/N),
$$
where $y_N^{\hat{\eta}}(\cdot)$ is the linear interpolation between the sequence of points we got in the Euler discretization of the reparametrized system.
\end{theorem}

\begin{proof}
Suppose that $N \in \mathcal{N}$, $\eta$ and $\hat{\eta}$ are given and $y^{\eta}(\cdot)$ and $y_N^{\hat{\eta}}(\cdot)$ are as above. Because of Picard Lemma we have
$$
|y_N^{\hat{\eta}}(s)-y^{\eta}(s)|\leq |\xi^0-\hat{\xi}^0|+\int_0^1|\dot{y}_N^{\hat{\eta}}(s)-h(y_N^{\hat{\eta}}(s),u(s),\Omega)|ds=:|\xi^0-\hat{\xi}^0| + e(y_N^{\hat{\eta}},\eta),
$$
where $h$ is defined as in Observation \ref{obs1}.

If we substitute the expressions of them we get
$$
e(y_N^{\hat{\eta}},\eta) \leq \sum_{k=0}^{N-1}\int_{s_k}^{s_{k+1}}\left[|f(y_N^{\hat{\eta}}(s_k),u(s_k))-f(y_N^{\hat{\eta}}(s),u(s_k))|\left|\frac{\theta(s_{k+1})-\theta(s_k)}{h}\right| \right]ds
$$
$$
+\sum_{k=0}^{N-1}\int_{s_k}^{s_{k+1}}|f(y_N^{\hat{\eta}}(s),u(s_k))|\left|\dot{\theta}(s)-\frac{\theta(s_{k+1})-\theta(s_k)}{h} \right|ds
$$
$$
+\sum_{k=0}^{N-1}\int_{s_k}^{s_{k+1}}\left[|g(y_N^{\hat{\eta}}(s_k))-g(y_N^{\hat{\eta}}(s))|\left|\frac{\phi(s_{k+1})-\phi(s_k)}{h}\right| \right]ds
$$
$$
+\sum_{k=0}^{N-1}\int_{s_k}^{s_{k+1}}|g(y_N^{\hat{\eta}}(s))|\left|\dot{\phi}(s)-\frac{\phi(s_{k+1})-\phi(s_k)}{h} \right|=:I+II+III+IV.
$$
We need to bound $I,II,III$ and $IV$.\\
For $I+IV$, as $f(\cdot,\cdot)$ is Lipschitz of rank $K^{'}$, $g(\cdot)$ is Lipschitz of rank $K^{''}$, $\theta(\cdot)$ is Lipschitz of rank $b$, and $\phi(\cdot)$ is Lipschitz of rank $r$, we have
$$
I+IV \leq \sum_{k=0}^{N-1}\int_{s_k}^{s_{k+1}} (bK^{'}+rK^{''})|y_N^{\hat{\eta}}(s_k)-y_N^{\hat{\eta}}(s)|ds.
$$
If we substitute the expression of $y_N^{\hat{\eta}}(s)$ and define $\rho:=bK^{'}+rK^{''}$,
$$
I+IV\leq \rho\sum_{k=0}^{N-1}\int_{s_k}^{s_{k+1}}|N(s-s_k)(y_N^{\hat{\eta}}(s_{k+1})-y_N^{\hat{\eta}}(s_k)|ds,
$$
that is,
$$
I+IV\leq \rho\sum_{k=0}^{N-1}\int_{s_k}^{s_{k+1}}N\beta(s-s_k)|hK_1(b+r)(|y_N^{\hat{\eta}}(s_{k})|+1)|ds.
$$
In the last inequality we just substituted the expression of $y_N^{\hat{\eta}}(s_{k+1})-y_N^{\hat{\eta}}(s_k)$ given by the Euler discretization, and we used the fact that $f(\cdot,\cdot)$ and $g(\cdot)$ have linear growth in the first variable.

Integrating,
$$
I+IV\leq \rho\beta\sum_{k=0}^{N-1}\frac{1}{2N^2}(|y_N^{\hat{\eta}}(s_{k})|+1) \leq \frac{\rho\beta \hat{K}_{\xi^0}}{2N},
$$
where $\hat{K}_{\xi^0}:=\sup\{|\hat{\xi}^0|+1: (\hat{\xi}^0,u,\Omega) \in S\}$, $\beta:=K_1(b+r)$ and in the last inequality we used Lemma (tese).

Now, define $\Phi_N:[0,1]\rightarrow \R$ as
$$
\Phi_N(s):= \max\left\{\left|\dot{\theta}(s)-\frac{\theta(s_{k+1})-\theta(s_k)}{h} \right|,\left|\dot{\phi}(s)-\frac{\phi(s_{k+1})-\phi(s_k)}{h}\right|\right\}.
$$
We know $\Phi_N(s) \rightarrow 0$ as $N \rightarrow \infty$. Then, given $\varepsilon=1/N$, there exists $N_0 \in \mathcal{N}$ so that for all $N_1 \geq N_0$ we have
\begin{equation}\label{31}
|\Phi_{N_1}(s)|\leq 1/N.
\end{equation}
If $N \geq N_0$, inequality \eqref{31} holds. If $N<N_0$ then define
$$
M_1:= \max_{n\in\{0,...,N_0\}}\{\Phi_n(s)\}.
$$
We have,
\begin{equation}\label{32}
|\Phi_N(s)|\leq M_1\leq \frac{M_1N_0}{N}=:\frac{M}{N}.
\end{equation}
For $II+III$, as $f(\cdot,\cdot)$ and $g(\cdot)$ have linear growth in the first variable, 
$$
II+III\leq \sum_{k=0}^{N-1}\int_{s_k}^{s_{k+1}}2K_1(|y_N^{\hat{\eta}}(s)|+1)\Phi_N(s)ds\leq \sum_{k=0}^{N-1}\int_{s_k}^{s_{k+1}}2K_1(\beta +1)e^{\beta}(1+|\hat{\xi}^0|)\Phi_N(s)ds,
$$
where in last inequality we substituted the expression of $y_N^{\hat{\eta}}(s)$ and used Lemma (tese). Then, 
$$
II+III\leq 2K_1(\beta+1)e^{\beta}(1+|\hat{\xi}^0|)\int_0^1\Phi_N(s)ds \leq \frac{2K_1(\beta+1)e^{\beta}\hat{K}_{\xi^0}M}{N},
$$
because of inequality \eqref{32}.

Now, define 
$$
K_{\xi^0}:= \max\{1,\frac{\rho\beta \hat{K}_{\xi^0}}{2},2K_1(\beta+1)e^{\beta}\hat{K}_{\xi^0}M\},
$$
then we have the result
$$
|y_N^{\hat{\eta}}(s)-y^{\eta}(s)|\leq K_{\xi^0}(|\xi^0-\hat{\xi}^0|+1/N).
$$
\end{proof}

The following theorem is given us the error between the objective function.

\begin{theorem}
Suppose that Assumption 1 holds. Then, for every bounded subset $S\subset B$, there exists a constant $K_S>0$ so that for any $\eta \in S \cap S_N$ and $k=0,...,N$,
$$
|f^0(\xi^0,y^{\eta}(s_k))-f^0(\xi^0,y_N^{\eta}(s_k))|\leq K_S/N,
$$
where $y^{\eta}(\cdot)$ is the solution of the reparametrized system and $y_N^{\eta}(\cdot)$ is the linear interpolation of the points we got in that Euler discretization.
\end{theorem}

\begin{proof}
As $f^0(\cdot,\cdot)$ is Lipschitz of rank $L_1$ we have
\[\begin{array}{lll}
|f^0(\xi^0,y^{\eta}(s_k))-f^0(\xi^0,y_N^{\eta}(s_k))|\leq L_1|y^{\eta}(s_k)-y_N^{\eta}(s_k)|
\leq L_1K_{\xi^0}/N=:K_S/N,
\end{array}\]
where in the last inequality we used the previous theorem.
\end{proof}

\section{Conclusions}
We study an impulsive optimal control problem in which we show we can get approximated problems to the reparametrized problem by Euler's discretization and if the sequence of approximated problems converge, it will converge to a solution of the reparametrized problem. We also show that a subsequence of discrete-time approximated functions graph-converges to a solution.  The results obtained come from a mix of ideas from consistent approximations given by \cite{Polak} and  Euler approximation and graph-convergence for impulsive differential inclusion given by \cite{Peter}. We are contributing with the literature about numerical methods for impulsive optimal control problems. 

\section{Acknowledgment}

Grant 2011/14121-9, 2014/05558-2 and 2013/07375-0, S\~ao Paulo Research Foundation (FAPESP).

\bibliographystyle{siam} 
\bibliography{exemplo}%

\end{document}